\documentclass[12pt]{amsart}

\usepackage{amsmath}
\usepackage{amssymb}
\usepackage{stmaryrd}
\usepackage{amscd}
\usepackage{calc}
\usepackage{ifthen}

\input prepictex   \input pictex    \input postpictex

\newcounter{boxsize}
\newcounter{tempcounter}

\newcommand{\smallentryformat}{\scriptstyle\sf}
\def\arr#1#2{\arrow <2mm> [0.25,0.75] from #1 to #2}
\def\num#1{$\scriptscriptstyle\sf#1$}
\def\smallsq#1{\plot 0 0  0.#1 0  0.#1 0.#1  0 0.#1  0 0 /}
\def\ssq{$\smallsq2$}
\newcommand\smbox{\put(0,0){\line(1,0){\value{boxsize}}}%
  \put(\value{boxsize},0){\line(0,1){\value{boxsize}}}%
  \put(0,0){\line(0,1){\value{boxsize}}}%
  \put(0,\value{boxsize}){\line(1,0){\value{boxsize}}}}

\newcommand\singlebox[1]{\raisebox{-.4ex}{\begin{picture}(4,0)\setcounter{boxsize}{4}%
    \put(0,0)\smbox%
    \put(0,0){\makebox(\value{boxsize},\value{boxsize})[c]{%
      $\sf#1$}}\end{picture}}}

\newcommand\boxes[2]{\setcounter{tempcounter}{#1*\value{boxsize}/2}
  \multiput(0,-\value{tempcounter})(0,\value{boxsize}){#1}\smbox
  \ifthenelse{#2=1}{\put(0,-\value{tempcounter}){%
      \makebox(\value{boxsize},\value{boxsize})[c]{$\smallentryformat1$}}}{}%
  \ifthenelse{#2=2}{\put(0,-\value{tempcounter}){%
      \makebox(\value{boxsize},\value{boxsize})[c]{$\smallentryformat2$}}}{}%
  \ifthenelse{#2=3}{\put(0,-\value{tempcounter}){%
      \makebox(\value{boxsize},\value{boxsize})[c]{$\smallentryformat2$}}%
    \setcounter{tempcounter}{\value{tempcounter}-\value{boxsize}}%
  \put(0,-\value{tempcounter}){\makebox(\value{boxsize},
    \value{boxsize})[c]{$\smallentryformat1$}}}{}
  \ifthenelse{#2=5}{\put(0,-\value{tempcounter}){%
      \makebox(\value{boxsize},\value{boxsize})[c]{$\smallentryformat3$}}%
    \setcounter{tempcounter}{\value{tempcounter}-\value{boxsize}}%
  \put(0,-\value{tempcounter}){\makebox(\value{boxsize},
    \value{boxsize})[c]{$\smallentryformat2$}}
    \setcounter{tempcounter}{\value{tempcounter}-\value{boxsize}}%
  \put(0,-\value{tempcounter}){\makebox(\value{boxsize},
    \value{boxsize})[c]{$\smallentryformat1$}}}{}
  \ifthenelse{#2=8}{\put(0,-\value{tempcounter}){%
      \makebox(\value{boxsize},\value{boxsize})[c]{$\smallentryformat4$}}%
    \setcounter{tempcounter}{\value{tempcounter}-\value{boxsize}}%
  \put(0,-\value{tempcounter}){\makebox(\value{boxsize},
    \value{boxsize})[c]{$\smallentryformat3$}}
    \setcounter{tempcounter}{\value{tempcounter}-\value{boxsize}}%
  \put(0,-\value{tempcounter}){\makebox(\value{boxsize},
    \value{boxsize})[c]{$\smallentryformat2$}}
    \setcounter{tempcounter}{\value{tempcounter}-\value{boxsize}}%
  \put(0,-\value{tempcounter}){\makebox(\value{boxsize},
    \value{boxsize})[c]{$\smallentryformat1$}}}{}
}
\newcommand\numbox[1]{\put(0,0)\smbox%
  \put(0,0){\makebox(\value{boxsize},\value{boxsize})[c]{%
      $\smallentryformat#1$}}}
\newcommand\vdotbox{\setcounter{tempcounter}{\value{boxsize}*2}
  \multiput(0,-\value{boxsize})(\value{boxsize},0)2{%
    \line(0,1){\value{tempcounter}}}
  \put(0,\value{boxsize}){\line(1,0){\value{boxsize}}}
  \put(0,-2){\makebox(\value{boxsize},\value{tempcounter})[c]{%
      $\scriptscriptstyle\vdots$}}}
\newcommand\rectbox[1]{\setcounter{tempcounter}{#1*\value{boxsize}}
  \put(0,0){\line(1,0){\value{boxsize}}}
  \put(0,\value{tempcounter}){\line(1,0){\value{boxsize}}}
  \put(0,0){\line(0,1){\value{tempcounter}}}
  \put(\value{boxsize},0){\line(0,1){\value{tempcounter}}}}

\DeclareMathOperator\type{{\rm type}}

\DeclareMathOperator\ann{{\rm ann}}

\DeclareMathOperator\incl{{\rm incl}}
\DeclareMathOperator\can{{\rm can}}
\DeclareMathOperator\Hom{{\rm Hom}}

\DeclareMathOperator\GL{{\rm GL}}
\DeclareMathOperator\Cok{{\rm Cok}}

\renewcommand\Im{{\rm Im}}
\DeclareMathOperator\len{{\rm len}}

\newtheoremstyle{mytheorems}{9pt}{6pt}{\itshape}{0pt}{\sc}{.}{ }{}
\newtheoremstyle{myremarks}{6pt}{3pt}{\normalfont}{0pt}{\it}{.}{ }{}

\theoremstyle{mytheorems}
\newtheorem{theorem}{Theorem}
\newtheorem{lemma}{Lemma}
\newtheorem{corollary}{Corollary}
\newtheorem{proposition}{Proposition}

\theoremstyle{myremarks}

\newtheorem*{acknowledgements}{Acknowledgement}

\newtheorem{remark}{Remark}
\newtheorem*{definition}{Definition}

\begin{document}
\title{The Entries in the LR-Tableau}
\author{Markus Schmidmeier}

\begin{abstract}
    Let $\Gamma$ be the Littlewood-Richardson 
    tableau corresponding to an embedding $M$ of
    a subgroup in a finite abelian $p$-group. 
    Each individual entry in $\Gamma$ yields information about 
    the homomorphisms from $M$ into a particular subgroup embedding,
    and hence determines the position of $M$ within the category
    of subgroup embeddings.  Conversely, this category
    provides a categorification for LR-tableaux in the sense that all
    subgroup embeddings corresponding to a given LR-tableau share certain
    homological properties.
\end{abstract}

\maketitle

\renewcommand{\thefootnote}{}

\bigskip
Let $\alpha$, $\beta$, $\gamma$ be partitions describing the isomorphism
types of finite abelian $p$-groups $A$, $B$, $C$.
We know from theorems by Green and Klein (\cite{klein}, \cite{macdonald})
that there is a subgroup embedding $M:(U\subset B)$ 
where $U\cong A$ and $B/U\cong C$ if and only if 
there is a Littlewood-Richardson tableau $\Gamma$ of type $(\alpha,\beta,\gamma)$.
The LR-coefficient $c_{\alpha,\gamma}^\beta$ counts the number of such
LR-tableaux and provides decisive information in surprisingly many areas of
modern algebra.  In particular, it determines the multiplication of Schur
functions, describes the decomposition of tensor products of irreducible 
polynomial representations of $\GL_n(\mathbb C)$, detects the possible eigenvalues
of sums of Hermitian matrices, and forms the highest coefficient of the 
classical Hall polynomial. 
Recent research (see e.~g.\ \cite{CDW}, \cite{kt}, \cite{cg}) uses 
modern methods in representation
theory to study growth and applications of LR-coefficients. 

\smallskip
It is the aim of this manuscript to give an
interpretation for each individual entry in the LR-tableau $\Gamma$
in terms of properties of the corresponding embedding $M$,
when considered as an object in the category of all subgroup embeddings.

\begin{definition}
Let $R$ be a (commutative) discrete valuation ring with maximal ideal $(p)$ 
and residue field $k=R/(p)$. Denote by 
$\mathcal S$ the category of all embeddings $(A\subset B)$ of (finite length) $R$-modules;
morphisms in $\mathcal S$ are given by commutative diagrams.  
Embeddings where $B$ is cyclic are  called {\it pickets;} they are
determined uniquely, up to isomorphism, by the lengths $\ell = \len A$ and
$m=\len B$:
$$P_\ell^m :\quad \left( (p^{m-\ell})\subset R/(p^m)\right).$$
\end{definition}

Our main result gives an interpretation of the entries in the LR-tableau
in terms of homomorphisms into pickets,  and in terms of the
picket decomposition of a suitable subquotient of the embedding.

\begin{theorem}\label{theorem-picket}
  For an embedding $M: (A\subset B)$ with LR-tableau $\Gamma$ and 
  natural numbers $\ell,m$ with $1\leq \ell\leq m$, the 
  following numbers are equal.
  \begin{enumerate}
    \item The number of boxes $\singlebox \ell$
      in the $m$-th row of $\Gamma$.
    \item The multiplicity of the picket $P_1^m$ 
      as a direct summand of the subfactor
      $$\left(\frac{p^{\ell-1}A}{p^\ell A}\subset \frac{B}{p^\ell A}\right).$$
    \item The dimension
      $$\dim_k \frac{\Hom_{\mathcal S}(M,P_\ell^m)}{\Im\Hom_{\mathcal S}(M,g_\ell^m)}.$$
  \end{enumerate}
\end{theorem}

The map $g_\ell^m$ is the sink map for $P_\ell^m$ in a suitable picket category. It is
defined as follows, with each component map an inclusion of pickets:
$$g_\ell^m:\left\{\begin{array}{rcll} 
                     P^m_{m-1} & \to & P_m^m & \quad\text{if}\; \ell = m \\
     P_{\ell-1}^m \oplus P_\ell^{m-1}  & \to & P_\ell^m & \quad\text{if}\; 1\leq \ell<m \\
            P_0^{m-1}  & \to & P_0^m & \quad\text{if}\; \ell = 0 \\
\end{array}\right. $$

\begin{remark}
$\Im\Hom(M,g^m_\ell)$ is the subgroup of $\Hom(M,P_\ell^m)$ of all maps which
factor through $g^m_\ell$, or equivalently, the subgroup generated by all maps
which factor through any proper inclusion of a picket in $P_\ell^m$.
\end{remark}

\begin{remark}
The category $\mathcal S$ of submodule embeddings provides a categorification 
for LR-tableaux in the sense that objects in $\mathcal S$ which correspond to the same
LR-tableau share the homological properties given by homomorphisms into pickets
(statement 3.\ in the Theorem). As we will see in the examples in 
Section~{\ref{section-examples}},
such objects are located in the same vicinity in the Auslander-Reiten quiver.
\end{remark}

\begin{definition}
For natural numbers $\ell, n$, let $\mathcal S_\ell$ and $\mathcal S(n)$ denote the full
subcategories of $\mathcal S$ of all pairs $(A\subset B)$ which satisfy $p^\ell A=0$ and
$p^n B=0$, respectively. 
\end{definition}

\begin{proposition}\label{prop-bilinear}
Let $1\leq \ell\leq m\leq n$ be natural numbers.
Define $C_\ell^m=\tau_{\mathcal S(n)}^{-1}(P^{m-1}_{\ell-1})$ if $\ell<m$, and 
$C_\ell^m=P_n^n$ if $\ell=m$. 
\begin{enumerate}
\item The factor $\displaystyle
         \frac{\Hom_{\mathcal S}(C_\ell^m,P_\ell^m)}{\Im\Hom_{\mathcal S}(C_\ell^m,g_\ell^m)}$ is
  a 1-di\-men\-sional $k$-vector space.
\item Suppose that $M$ is an embedding in $\mathcal S(n)$. 
The bilinear form given by composition,
  $$\phantom.\qquad\frac{\Hom_{\mathcal S}(M,P_\ell^m)}{\Im\Hom_{\mathcal S}(M,g_\ell^m)}\times 
        \Hom_{\mathcal S}(C_\ell^m,M)\longrightarrow
  \frac{\Hom_{\mathcal S}(C_\ell^m,P_\ell^m)}{\Im\Hom_{\mathcal S}(C_\ell^m,g_\ell^m)}$$
  is left non-degenerate.
\end{enumerate} 
\end{proposition}

As a consequence, in each category $\mathcal S(n)$, the embeddings $M$ which have
an entry $\singlebox \ell$ in the $m$-th row of their LR-tableau are characterized
by admitting maps $C_\ell^m\to M\to P_\ell^m$ such that the composition
does not factor through $g_\ell^m$. In this sense, $M$ is located between
$C_\ell^m$ and $P_\ell^m$. 

\begin{acknowledgements}
The results in this paper have been obtained during the academic year 2008/09
when the author was visiting the Collaborative Research Center 701 at Bielefeld University,
and then the Mathematical Institute at Cologne University.
He would like to thank Claus Michael Ringel and Steffen K\"onig for invitation 
and helpful advice.

\smallskip
The author also wishes to thank the referee,
in particular for suggesting a significant simplification
for the proof of Theorem~\ref{theorem-picket}.
\end{acknowledgements}

\section{Partitions}
%----------------------

For $m$ a natural number, let $P^m$ be the cyclic $R$-module $R/(p^m)$. 
Recall that an arbitrary $R$-module is determined uniquely, up to isomorphism,
by a partition:

\begin{eqnarray*}
\big\{ p\text{-modules} \big\}{\big/}_{\textstyle\cong}  & 
          \quad \stackrel{1-1}{\longleftrightarrow}\quad & 
          \big\{\text{partitions}\big\}\\
    B  & \longmapsto & 
          \type(B)=\beta\\
       &  & \quad\text{if}\quad \beta'_i=\dim_k\big(p^{i-1}B/p^iB\big) \\
M(\beta)=\bigoplus\nolimits_{i=1}^s P^{\beta_i} & \longmapsfrom & 
          \beta=(\beta_1,\ldots,\beta_s)\\
\end{eqnarray*}

\smallskip
Here $\beta'$ denotes the transpose of the partition $\beta$.
Let $\beta,\gamma$ be partitions. We say $\gamma\leq\beta$ if 
$\gamma_i\leq\beta_i$ holds for each $i$; in this case $\beta$ and
$\gamma$ define a {\it skew tableau} $\beta-\gamma$. 
The skew tableau is
a {\it horizontal strip} if $\beta_i-\gamma_i\leq 1$ holds for 
each $i$, and the length $|\beta-\gamma|$ of the strip is given by the
sum $\sum_i(\beta_i-\gamma_i)$.

\begin{definition}
  An increasing sequence of partitions $\Gamma=[\gamma^0,\ldots,\gamma^s]$ defines an
  {\it LR-tableau of type} $(\alpha,\beta,\gamma)$ if $s=\alpha_1$, 
  $\gamma^0=\gamma$, $\gamma^s=\beta$
  and the following two conditions are satisfied.
  \begin{enumerate} 
  \item  For each $\ell\geq 1$, the skew tableau $\gamma^\ell-\gamma^{\ell-1}$
    is a horizontal strip of length $|\gamma^\ell-\gamma^{\ell-1}|=\alpha'_\ell$.
  \item The {\it lattice permutation property} is satisfied, that is for each $\ell\geq 2$ 
    and each $h$ we have
    $$ \sum_{i\geq h}(\gamma_i^\ell-\gamma_i^{\ell-1}) \leq 
    \sum_{i\geq h}(\gamma_i^{\ell-1}-\gamma_i^{\ell-2}).$$
  \end{enumerate}
\end{definition}

\smallskip
As usual, we picture an LR-tableau $\Gamma=[\gamma^0,\ldots,\gamma^s]$ 
by labelling each box in the skew tableau $\gamma^\ell-\gamma^{\ell-1}$ by $\singlebox \ell$.
Note that some authors replace the partitions in the tableau by their
transposes; here we follow the convention in \cite{klein}.

\smallskip
Given an embedding $M:(A\subset B)$ we obtain an LR-tableau as follows.
Put $\alpha=\type(A)$, $\beta=\type(B)$, $\gamma=\type(B/A)$,
let $s=\alpha_1$ be the Loewy length of $A$, and
define for each $0\leq \ell\leq s$ the partition 
$\gamma^\ell=\type(B/p^\ell A)$.  
According to Green's Theorem~\cite[Theorem~4.1]{klein}, 
the sequence $\Gamma=[\gamma^0,\ldots,\gamma^s]$ 
forms an LR-tableau of type $(\alpha,\beta,\gamma)$;
we say $\Gamma$ is {\it the LR-tableau for} $M$.

\section{Semisimple submodules}
%---------------------------------

The embeddings $(A\subset B)$ of $p$-mo\-dules where $A$ is semisimple
form the category $\mathcal S_1$, they are well understood:
The indecomposable embeddings are pickets $P_\ell^m$ with $\ell\leq 1$,
and arbitrary embeddings are given by their LR-tableaux \cite[Proposition 3.3]{klein}:

\begin{proposition} \label{prop-elementary}
The following map defines a one-to-one correspondence between the isomorphism
types of embeddings in $\mathcal S_1$ and horizontal strips:
\begin{eqnarray*}
\big\{\text{objects $(A\subset B)\in\mathcal S_1$}\big\}{\big/}_{\textstyle\cong} &
   \stackrel{1-1}{\longleftrightarrow} &
   \left\{\text{horizontal strips}\right\} \\
(A\subset B) & \longmapsto & \type(B)-\type(B/A) \\
\bigoplus P^{\beta_i}_{\beta_i-\gamma_i}  & \longmapsfrom & \beta-\gamma\\
\end{eqnarray*}

\vspace{-6ex}\qed
\end{proposition}

\begin{corollary} \label{corollary-skew-tableau}
Suppose the embedding $(A\subset B)\in\mathcal S_1$ 
corresponds to the skew tableau $\beta-\gamma$.
The number of boxes in the $m$-th row in $\beta-\gamma$ is equal to the 
multiplicity of the picket $P_1^m$ as a direct summand of $(A\subset B)$. 
\end{corollary}

\begin{proof}
The number of boxes in the $m$-th row in $\beta-\gamma$ is equal to
$$\beta'_m-\gamma'_m\;=\;\#\{\;i\;|\;\beta_i=m\;\text{and}\;\gamma_i=m-1\};$$
under the correspondence given above, this is the number of summands $P_1^m$
in the direct sum decomposition for $(A\subset B)$. 
\end{proof}

\smallskip
As a category, $\mathcal S_1$ is an exact Krull-Remak-Schmidt category which has
Aus\-lan\-der-Reiten sequences, see \cite[Section 3.1]{s-bounded}.  The Auslander-Reiten
quiver consists of one tube that has 2 rays but only 1 coray;
in this quiver we represent the indecomposable objects by their LR-tableaux.

$$
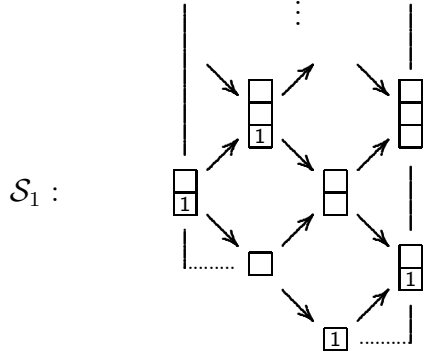
\setcounter{boxsize}{3}
\beginpicture\setcoordinatesystem units <1cm,1cm>
\put {} at 0 4.6
\put {} at 0 -.6
\put{$\mathcal S_1:$} at -2 2
\put{\begin{picture}(3,0)\boxes21\end{picture}} at 0 2
\put{\begin{picture}(3,0)\boxes10\end{picture}} at 1 1
\put{\begin{picture}(3,0)\boxes31\end{picture}} at 1 3
\put{\begin{picture}(3,0)\boxes11\end{picture}} at 2 0
\put{\begin{picture}(3,0)\boxes20\end{picture}} at 2 2
\put{\begin{picture}(3,0)\boxes21\end{picture}} at 3 1
\put{\begin{picture}(3,0)\boxes30\end{picture}} at 3 3
\arr{.3 1.7}{.7 1.3}
\arr{.3 2.3}{.7 2.7}
\arr{.3 3.7}{.7 3.3}
\arr{1.3 .7}{1.7 .3}
\arr{1.3 1.3}{1.7 1.7}
\arr{1.3 2.7}{1.7 2.3}
\arr{1.3 3.3}{1.7 3.7}
\arr{2.3 .3}{2.7 .7}
\arr{2.3 1.7}{2.7 1.3}
\arr{2.3 2.3}{2.7 2.7}
\arr{2.3 3.7}{2.7 3.3}
\setsolid
\plot 0 1  0 1.5 /
\plot 0 2.5  0 4.5 /
\plot 3 0  3 .5 /
\plot 3 1.5  3 2.35 /
\plot 3 3.65 3 4.5 /
\setdots<2pt>
\plot 0 1  .65 1 /
\plot 2.35 0  3 0 /
\put{\vdots} at 1.5  4.5 
\endpicture
$$

The maps $g_1^m$ from the introduction occur as sink maps of the
pickets of type $P_1^m$.

\section{Categorification}

In this section we show Theorem~\ref{theorem-picket}.

\smallskip
For $M=(A\subset B)$ an embedding and $\ell$ a natural number let
$$M|^\ell_1\;=\; \left(\frac{p^{\ell-1}A}{p^\ell A}\subset \frac B{p^\ell A}\right)$$
be the reduced embedding, which is an object in $\mathcal S_1$.  
The following result is clear from the definition.

\begin{lemma}
Suppose $M\in\mathcal S$ has LR-tableau $\Gamma=[\gamma^{0},\ldots,\gamma^{s}]$.
For $1\leq \ell\leq s$, the LR-tableau for the reduced embedding $M|^\ell_1$ is 
the skew tableau
$$\Gamma|^\ell_1\;=\; [\gamma^{\ell-1},\gamma^\ell].$$ \qed
\end{lemma}

Combining this result with Corollary~\ref{corollary-skew-tableau}
we obtain the equality of the numbers in (1) and (2) in Theorem~\ref{theorem-picket}:

\begin{corollary}
Suppose $M\in\mathcal S$ has LR-tableau $\Gamma$.  Let $1\leq\ell\leq m$. 
The number of boxes $\singlebox\ell$ in row $m$ in $\Gamma$ equals the
multiplicity of $P_1^m$ in a direct sum decomposition for $M|_1^\ell$.    \qed
\end{corollary}

For the proof of Theorem~\ref{theorem-picket}
we use the following three observations.

\smallskip
{\it Observation 1.} The number of boxes $\singlebox\ell$ in row $m$ 
in the LR-tableau $\Gamma=[\gamma^0,\ldots,\gamma^s]$ is
$$(\gamma^\ell)'_m - (\gamma^{\ell-1})'_m.$$

\smallskip
{\it Observation 2.} Suppose that an $R$-module $B$ has type $\beta$. 
Then the length of the $m$-th row in the partition $\beta$ is
$$\beta'_m \;=\; \len\Hom_R(B,P^m) - \len \Hom_R(B,P^{m-1}).$$

\smallskip
{\it Observation 3.} Homomorphisms into pickets can be expressed in terms
of homomorphisms between $R$-modules:
$$\Hom_{\mathcal S}( (A\subset B), P^m_\ell ) \;=\; \Hom_R ( B/p^\ell A, P^m )$$

\begin{proof}[Proof of Theorem \ref{theorem-picket}]
It remains to show the equality of the numbers in (1) and (3).
Suppose that the embedding $M:(A\subset B)$ has LR-tableau $\Gamma=[\gamma^0,\ldots,\gamma^s]$.

\smallskip
The number of boxes $\singlebox\ell$ in the $m$-th row in $\Gamma$ is
$(\gamma^\ell)'_m-(\gamma^{\ell-1})'_m$;
this number can be expressed in terms of $R$-homomorphisms as
\begin{eqnarray*} (*) & &  \len\Hom_R(B/p^\ell A, P^m) - \len\Hom_R(B/p^\ell A,P^{m-1})\\
  & & \qquad       -\len\Hom_R(B/p^{\ell-1}A, P^m) + \len\Hom_R(B/p^{\ell-1}A, P^{m-1})
\end{eqnarray*}
(Observation 2).

\smallskip
We first consider the case where $\ell<m$. 
Then $(*)$ can be written in terms of homomorphisms into pickets as 
\begin{eqnarray*} & & \len\Hom_{\mathcal S}((A\subset B),P^m_\ell) 
 - \len\Hom_{\mathcal S}((A\subset B),P^{m-1}_\ell)\\
 & & \qquad 
   -\len\Hom_{\mathcal S}((A\subset B),P^m_{\ell-1}) 
      + \len\Hom_{\mathcal S}((A\subset B),P^{m-1}_{\ell-1})
\end{eqnarray*}
which equals
$$\len\Cok\Hom_{\mathcal S}((A\subset B), g_\ell^m)$$
where $g_\ell^m$ is the epimorphism in the short exact sequence
$$0\longrightarrow P_{\ell-1}^{m-1}\longrightarrow P_{\ell-1}^m\oplus P_\ell^{m-1}
\stackrel{g_\ell^m}\longrightarrow P_\ell^m\longrightarrow 0.$$

\smallskip
In the remaining case where $\ell=m$, the expression $(*)$ simplifies to
$$ \len\Hom_R(B/p^\ell A, P^m) - \len\Hom_R(B/p^{\ell-1}A, P^m),$$
and hence to 
$$\len\Hom_{\mathcal S}((A\subset B),P^m_m) - \len\Hom_{\mathcal S}((A\subset B),P^m_{m-1})$$
which equals $\len\Cok\Hom_{\mathcal S}((A\subset B),g_m^m)$: For this apply
the functor $\Hom_{\mathcal S}((A\subset B),-)$ to the monomorphism
$g_m^m:P_{m-1}^m\to P_m^m$. 
\end{proof}

\begin{remark} Together, the maps of type $g_\ell^m$ 
form the sink  maps in the Auslander-Reiten quiver
for the category $\mathcal P{\it ic}$ which has as objects the direct sums of pickets;
morphisms are those maps for which each component is either zero or an
inclusion between pickets.
$$
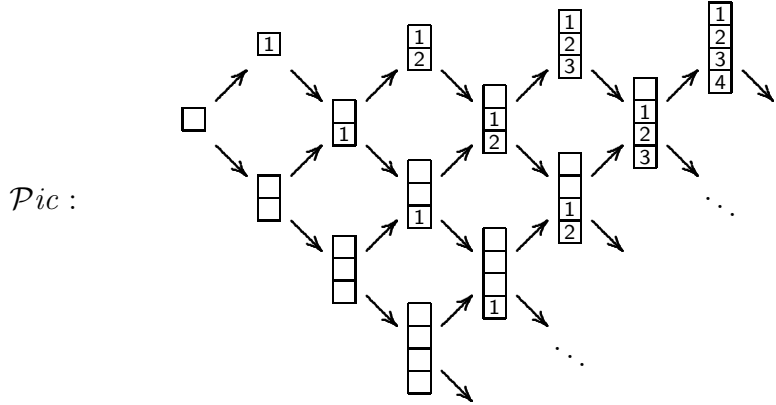
\setcounter{boxsize}{3}
\beginpicture\setcoordinatesystem units <1cm,1cm>
\put {} at 0 4.9
\put {} at 0 -1
\put{$\mathcal P{\it ic}:$} at -2 2
\put{\begin{picture}(3,0)\boxes10\end{picture}} at 0 3
\put{\begin{picture}(3,0)\boxes11\end{picture}} at 1 4
\put{\begin{picture}(3,0)\boxes20\end{picture}} at 1 2
\put{\begin{picture}(3,0)\boxes21\end{picture}} at 2 3
\put{\begin{picture}(3,0)\boxes30\end{picture}} at 2 1
\put{\begin{picture}(3,0)\boxes23\end{picture}} at 3 4
\put{\begin{picture}(3,0)\boxes31\end{picture}} at 3 2
\put{\begin{picture}(3,0)\boxes40\end{picture}} at 3 0
\put{\begin{picture}(3,0)\boxes33\end{picture}} at 4 3
\put{\begin{picture}(3,0)\boxes41\end{picture}} at 4 1
\put{\begin{picture}(3,0)\boxes35\end{picture}} at 5 4
\put{\begin{picture}(3,0)\boxes43\end{picture}} at 5 2
\put{\begin{picture}(3,0)\boxes45\end{picture}} at 6 3
\put{\begin{picture}(3,0)\boxes48\end{picture}} at 7 4
\arr{.3 3.3}{.7 3.7}
\arr{.3 2.7}{.7 2.3}
\arr{1.3 3.7}{1.7 3.3}
\arr{1.3 2.3}{1.7 2.7}
\arr{1.3 1.7}{1.7 1.3}
\arr{2.3 3.3}{2.7 3.7}
\arr{2.3 2.7}{2.7 2.3}
\arr{2.3 1.3}{2.7 1.7}
\arr{2.3 .7}{2.7 .3}
\arr{3.3 3.7}{3.7 3.3}
\arr{3.3 2.3}{3.7 2.7}
\arr{3.3 1.7}{3.7 1.3}
\arr{3.3 .3}{3.7 .7}
\arr{3.3 -.3}{3.7 -.7}
\arr{4.3 3.3}{4.7 3.7}
\arr{4.3 2.7}{4.7 2.3}
\arr{4.3 1.3}{4.7 1.7}
\arr{4.3 .7}{4.7 .3}
\arr{5.3 3.7}{5.7 3.3}
\arr{5.3 2.3}{5.7 2.7}
\arr{5.3 1.7}{5.7 1.3}
\arr{6.3 3.3}{6.7 3.7}
\arr{6.3 2.7}{6.7 2.3}
\arr{7.3 3.7}{7.7 3.3}
\multiput{$\ddots$} at 5 0  7 2 /
\endpicture
$$
\end{remark}

In the Remark under Theorem~\ref{theorem-picket} we claim that
$\Im\Hom(M,g_\ell^m)$ is the submodule of $\Hom(M,P_\ell^m)$ generated
by all maps which factor through some proper inclusion of a picket
in $P_\ell^m$.  Since any such inclusion factors through $g_\ell^m$, 
the claim follows.

\section{Intervals in the Auslander-Reiten quiver}

We show Proposition~\ref{prop-bilinear}.
Let $1\leq\ell\leq m\leq n$.  Define $C_\ell^m=\tau_{\mathcal S(n)}^{-1}(P_{\ell-1}^{m-1})$
if $\ell<m$ and $C_\ell^m=P_n^n$ for $\ell=m$. 
We characterize the indecomposable objects $M\in\mathcal S(n)$ which have a box
$\singlebox\ell$ in the $m$-th row of their LR-tableau 
by the existence of maps $f_1:C_\ell^m\to M$, $f_2:M\to P_\ell^m$ 
such that $f_2f_1$ does not factor through $g_\ell^m$.
In this sense, $M$ occurs in the interval from $C_\ell^m$ to $P_\ell^m$ 
in the Auslander-Reiten quiver for $\mathcal S(n)$.

\smallskip
First we determine $C_\ell^m$ in the case where $\ell < m$ and compute its LR-tableau.
We use \cite[Theorem~5.2]{rs-art} in which the Auslander-Reiten translation is
computed for indecomposable objects in the factor module category $\mathcal F(n)$.
The kernel and cokernel functors induce an equivalence between the categories 
$\mathcal F(n)$ and $\mathcal S(n)$ \cite[Lemma~1.2~(3)]{rs-art},
so we can compute $\tau_{\mathcal S(n)}^{-1}(P_{\ell-1}^{m-1})$ 
as the kernel of the minimal epimorphism representing $P_{\ell-1}^{m-1}$.

\smallskip
For the inclusion $P_{\ell-1}^{m-1}:(P^{\ell-1}\stackrel{\incl}{\subset} P^{m-1})$
the minimal epimorphism is 
$$P^n\oplus P^{\ell-1}\; \stackrel{({\rm can},\incl)}{\longrightarrow}\; P^{m-1},$$
and the map $(\incl, -\can):P^{n+\ell-m}\to P^n\oplus P^{\ell-1}$
is a kernel. This monomorphism represents the object $C_\ell^m$ in $\mathcal S(n)$. 

\smallskip
We determine the LR-tableau $\Gamma$ for $C_\ell^m=(A\subset B)$. 
Let $\alpha=\type(A)=(n+\ell-m)$, $\beta=\type(B)=(n,\ell-1)$, and $\gamma=\type(B/A)=(m-1)$.
Since there is only one filling of $\beta-\gamma$ in which each number
$1,\ldots,n+\ell-m$ occurs exactly once, the LR-tableau must be as follows:

\begin{center}
\setcounter{boxsize}{3}
  \begin{picture}(20,33)
    \put(-10,13){$\Gamma:$}
    \put(10,0){\begin{picture}(3,12)\put(0,3){\numbox t} 
        \put(0,9)\vdotbox 
        \put(0,12){\numbox \ell}
        \put(3,21){\numbox{\ell'}}
        \put(3,27)\vdotbox
        \put(3,30){\numbox 1}
        \put(-12,23){${}_{m-1}\left\{\makebox(0,9){}\right.$}
        \put(0,15){\rectbox6}
      \end{picture}}
  \end{picture}
\end{center}

Here, $t=n+\ell-m$ and $\ell'=\ell-1$. The LR-tableau for the picket $C_m^m=P_n^n$
is as follows:

\begin{center}
\setcounter{boxsize}{3}
  \begin{picture}(20,12)
    \put(-10,5){$\Gamma:$}
    \put(10,0){\begin{picture}(3,12)\put(0,0){\numbox n} 
        \put(0,6)\vdotbox 
        \put(0,9){\numbox 1}
      \end{picture}}
  \end{picture}
\end{center}

We can now complete the 
proof of Proposition~\ref{prop-bilinear}.

\begin{proof}  Let $1\leq \ell\leq m\leq n$.

1. In each case, 
there is exactly one box $\singlebox \ell$ in row $m$ in the above LR-tableau 
for $C_\ell^m$,
so $\dim\Cok\Hom(C_\ell^m,g_\ell^m)=1$, by Theorem~\ref{theorem-picket}.

\smallskip
2. We consider first the case where $\ell<m$.  Let $\mathcal A$ be the sequence
$$\mathcal A: \quad 
       0\to P_{\ell-1}^{m-1}\to P_{\ell-1}^m\oplus P_\ell^{m-1}\stackrel{g_\ell^m}\to P_\ell^m\to 0.$$
and let 
$$ \mathcal E:\quad 0\to P_{\ell-1}^{m-1}\to B\to C\to 0$$
be the Auslander-Reiten sequence in $\mathcal S(n)$ starting at $P_{\ell-1}^{m-1}$,
its end term is $C=C_\ell^m$. 
Since $\mathcal A$ is nonsplit, there are maps 
$h':B\to P_{\ell-1}^m\oplus P_\ell^{m-1}$, $h:C\to P_\ell^m$ which make the upper part
of the following diagram commutative.
$$
\begin{CD}
\mathcal E:\quad @.  0 @>>> P_{\ell-1}^{m-1} @>u>> B @>v>> C @>>> 0 \\
@. @.  @|   @V{h'}VV  @VV{h}V   \\
\mathcal A:\quad @.  0 @>>> P_{\ell-1}^{m-1} @>{f_\ell^m}>> P_{\ell-1}^m\oplus P_\ell^{m-1} @>{g_\ell^m}>> P_\ell^m @>>> 0\\
@.  @.  @|   @A{r'}AA @AA{r}A \\
r\mathcal A:\quad @. 0 @>>> P_{\ell-1}^{m-1} @>{s}>> L @>{t}>> M @>>> 0 \\
\end{CD}
$$
Suppose that $M\in\mathcal S(n)$.
In order to show that the bilinear form given by composition 
$$\frac{\Hom(M,P_\ell^m)}{\Im\Hom(M,g_\ell^m)}\times \Hom(C,M)\longrightarrow
  \frac{\Hom(C,P_\ell^m)}{\Im\Hom(C,g_\ell^m)}$$
is left non-degenerate,
let $r:M\to P_\ell^m$ be a map which does not factor through $g_\ell^m$. 
We will construct $q:C\to M$ such that $rq$ does not factor through $g_\ell^m$.
Since $r$ does not factor through $g_\ell^m$, the induced sequence at the bottom 
of the above diagram does not split. Hence the map $s$ factors through $u$:
There is a map $q':B\to L$ such that $s=q'u$.  Let $q:C\to M$ be the cokernel map,
so $qv=tq'$.  Then $rqv=rtq'=g_\ell^mr'q'$.  Since $r'q'u=r's=f_\ell^m=h'u$,
there exists $z:C\to P_{\ell-1}^m\oplus P_\ell^{m-1}$ such that $zv=r'q'-h'$.
So $rqv=g_\ell^mr'q'=g_\ell^m(zv+h')=(g_\ell^m z+h)v$ and since $v$ is onto,
$rq=g_\ell^mz +h$.  Since $\mathcal E$ is not split exact, $h$ does not factor through
$g_\ell^m$, and hence  $rq$ does not factor over $g_\ell^m$.

\smallskip
We deal with the case where $\ell=m$. 
Let $f_2:M\to P_m^m$ be a map which does not factor
through the inclusion $g_m^m:P_{m-1}^m\to P_m^m$ of the maximal submodule, 
so $f_2$ is an epimorphism.
Since $P_n^n$ is projective --- even in the abelian category of all maps between
$R/(p^n)$-modules --- the canonical map $P_n^n\to P_m^m$ factors through $f_2$.
\end{proof}

\section{Duality}

Let $I$ be the injective envelope of the simple $R$-module $P^1$.  Then the functor
$*=\Hom_R(-,I)$ defines a duality for $R$-modules, which gives rise to a duality
for short exact sequences of $R$-modules, and hence yields a duality on $\mathcal S$.  

\smallskip
We have seen that the entries in the LR-tableau for an object $M\in\mathcal S$ 
are isomorphism invariants for $M$ and can be interpreted in terms of homomorphisms
from $M$ into pickets.
Also the entries in the LR-tableau for $M^*$ are isomorphism invariants for $M$;
they have the following interpretation in terms of homomorphisms from pickets into $M$.

\begin{theorem} \label{theorem-picket-dual}
  For an embedding $M=(A\subset B)\in\mathcal S$ and $1\leq \ell\leq m$, the 
  following numbers are equal.
  \begin{enumerate}
    \item The number of boxes $\singlebox\ell$ 
      in the $m$-th row of the LR-tableau for the dual module $M^*$.
    \item The multiplicity of $P_1^m$ as a direct summand of
      $$\left(\frac{p^{\ell-1}U}{p^\ell U}\subset \frac{B^*}{p^\ell U}\right)$$
      where $U=\ann_{B^*}A$ is the annihilator of $A$ in $B^*$.
    \item The dimension
      $$\dim_k\frac{\Hom_{\mathcal S}(P^m_{m-\ell},M)}{\Im\Hom_{\mathcal S}(h^m_{m-\ell},M)}.$$
  \end{enumerate}
\end{theorem}

For $1\leq q\leq m$, the morphism $h_q^m$ 
is the following map between sums of pickets;
each component map is an epimorphism on the total space and on the factor space.
$$h_q^m\;:\;\left\{\begin{array}{rcll}
            P^m_0 & \to & P^m_1, & \text{if}\; q=0 \\
            P^m_q & \to & P_{q-1}^{m-1}\oplus P_{q+1}^m, & \text{if}\; 1\leq q<m\\
            P_m^m & \to & P^{m-1}_{m-1}, & \text{if}\; q=m\\
      \end{array}\right.$$

\begin{proof}
The result is an easy consequence of Theorem \ref{theorem-picket}:

\smallskip
For the equality of the numbers in (1) and (2) it suffices to note 
that if $M$ is the embedding $(A\subset B)$, then the dual $M^*$ 
is the embedding $(U\subset B^*)$ where $U=\ann_{B^*}A$.

\smallskip
We show the equality of the numbers in (1) and (3)  
According to Theorem~\ref{theorem-picket},
the number in (1) is the dimension
$$\dim_k\frac{\Hom(M^*,P_\ell^m)}{\Im\Hom(M^*,g_\ell^m)},$$
which is equal to the dimension
$$\dim_k\frac{\Hom(P^m_{m-\ell},M)}{\Im\Hom(h_{m-\ell}^m,M)}$$
since $h_{m-\ell}^m$ is the dual of the map $g_\ell^m$.
\end{proof}

With the following result we can position $M$ within the category $\mathcal S$:

\begin{proposition}\label{prop-bilinear-dual}
  Suppose that $0\leq q <m\leq n$. 
  Let $A_q^m=\tau_{\mathcal S(n)}P_q^{m-1}$ if $q>0$ and $A^m_q=P^n_0$ if $q=0$.
  \begin{enumerate}
  \item The factor $\displaystyle\frac{\Hom(P_q^m,A_q^m)}{\Im\Hom(h_q^m,A_q^m)}$ 
    is a 1-dimensional $k$-vector space.
  \item For $M\in\mathcal S(n)$, the bilinear form given by composition,
    $$\qquad\Hom_{\mathcal S}(M,A_q^m)\times
    \frac{\Hom_{\mathcal S}(P^m_q,M)}{\Im\Hom_{\mathcal S}(h^m_q,M)}
    \longrightarrow
    \frac{\Hom_{\mathcal S}(P^m_q,A_q^m)}{\Im\Hom_{\mathcal S}(h^m_q,A_q^m)}$$
    is right non-degenerate.
  \end{enumerate}
\end{proposition}

\begin{proof}
This result follows from Proposition~\ref{prop-bilinear} by duality.
\end{proof}

\section{Example: Submodules of $p^5$-bounded modules.} \label{section-examples}
%---------------------------------------------------------

The results in this paper can be visualized on the Auslander-Reiten quivers
of categories of embeddings.  In this section we consider the category $\mathcal S(5)$, 
which among all the categories
of type $\mathcal S(n)$ is the largest of finite representation type
\cite{rs1}.  
There are 50 indecomposable objects in $\mathcal S(5)$; we picture here  
the Auslander-Reiten quiver $\Gamma_{\mathcal S(5)}$ 
from \cite[(6.5)]{rs1}, with the objects represented
by their LR-tableaux.

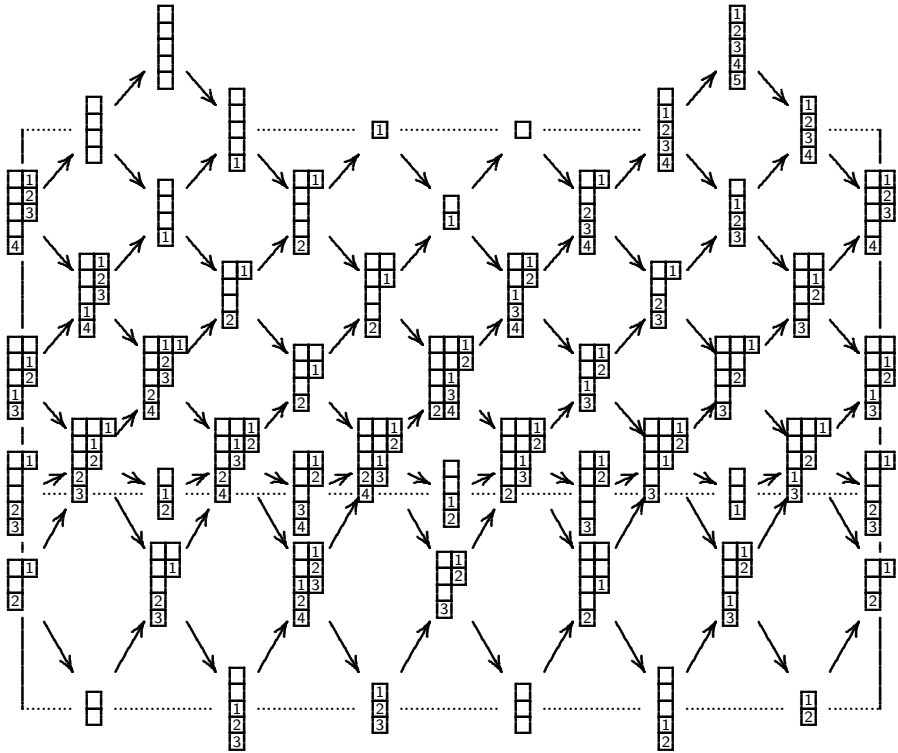
\def\GammaFive{\multiput{\ssq} at 1.9 -.1  1.9 .1 /  
\multiput{\ssq} at 3.9 -.4  3.9 -.2  3.9 0  3.9 .2  3.9 .4 /
\put{\num1} at 4 0
\put{\num2} at 4 -.2
\put{\num3} at 4 -.4
\multiput{\ssq} at 5.9 -.2  5.9 0  5.9 .2 /
\put{\num1} at 6 .2
\put{\num2} at 6 0
\put{\num3} at 6 -.2
\multiput{\ssq} at 7.9 -.2  7.9 0  7.9 .2 /
\multiput{\ssq} at 9.9 -.4  9.9 -.2  9.9 0  9.9 .2  9.9 .4 /
\put{\num1} at 10 -.2
\put{\num2} at 10 -.4
\multiput{\ssq} at 11.9 -.1  11.9 .1 /
\put{\num1} at 12 .1
\put{\num2} at 12 -.1 
\multiput{\ssq} at .8 1.3  .8 1.5  .8 1.7  1 1.7 /
\put{\num1} at 1.1 1.7
\put{\num2} at .9 1.3
\multiput{\ssq} at  2.8 1.1  2.8 1.3  2.8 1.5  2.8 1.7  2.8 1.9  3 1.7  3 1.9 /
\put{\num1} at 3.1 1.7
\put{\num2} at 2.9 1.3
\put{\num3} at 2.9 1.1
\multiput{\ssq} at 4.8 1.1  4.8 1.3  4.8 1.5  4.8 1.7  4.8 1.9  
                5 1.9  5 1.5  5 1.7 /
\multiput{\num1} at 4.9 1.5  5.1 1.9 /
\multiput{\num2} at 4.9 1.3  5.1 1.7 /
\put{\num4} at 4.9 1.1
\put{\num3} at 5.1 1.5
\multiput{\ssq} at 6.8 1.2  6.8 1.4  6.8 1.6  6.8 1.8  7 1.6  7 1.8 /
\put{\num1} at 7.1 1.8
\put{\num2} at 7.1 1.6
\put{\num3} at 6.9 1.2
\multiput{\ssq} at  9 1.5  9 1.7  9 1.9  
                8.8 1.1  8.8 1.3  8.8 1.5  8.8 1.7  8.8 1.9 /
\put{\num1} at 9.1 1.5
\put{\num2} at 8.9 1.1 
\multiput{\ssq} at 10.8 1.1  10.8 1.3  10.8 1.5  10.8 1.7  10.8 1.9  
                                11 1.7  11 1.9 /
\multiput{\num1} at 11.1 1.9  10.9 1.3 /
\put{\num2} at 11.1 1.7 
\put{\num3} at 10.9 1.1 
\multiput{\ssq} at 12.8 1.3  12.8 1.5  12.8 1.7  13 1.7 /
\put{\num1} at 13.1 1.7
\put{\num2} at 12.9 1.3
%
% zentrale orbiten
%
\multiput{\ssq} at .8 2.2  .8 2.4  .8 2.6  .8 2.8  .8 3  1 3 /
\put{\num1} at 1.1 3
\put{\num2} at .9 2.4
\put{\num3} at .9 2.2
\multiput{\ssq} at 1.7 2.6  1.7 2.8  1.7 3  1.7 3.2  1.7 3.4  
                1.9 3.4 1.9 3  1.9 3.2  2.1 3.4 /
\multiput{\num1} at 2 3.2  2.2 3.4 /
\multiput{\num2} at 2 3  1.8 2.8 /
\put{\num3} at 1.8 2.6
\multiput{\ssq} at 2.9 2.4  2.9 2.6  2.9 2.8 /
\put{\num1} at 3 2.6
\put{\num2} at 3 2.4
\multiput{\ssq} at 3.7 2.6  3.7 2.8  3.7 3  3.7 3.2  3.7 3.4  
                3.9 3.4  3.9 3  3.9 3.2   4.1 3.4  4.1 3.2 /
\multiput{\num1} at 4.2 3.4  4 3.2 /
\multiput{\num2} at 4.2 3.2  3.8 2.8 /
\put{\num3} at 4 3
\put{\num4} at 3.8 2.6
\multiput{\ssq} at 4.8 2.2  4.8 2.4  4.8 2.6  4.8 2.8  4.8 3  5 2.8  5 3 /
\put{\num1} at 5.1 3
\put{\num2} at 5.1 2.8
\put{\num3} at 4.9 2.4
\put{\num4} at 4.9 2.2
\multiput{\ssq} at 5.7 2.6  5.7 2.8  5.7 3  5.7 3.2  5.7 3.4  
                5.9 2.8  5.9 3  5.9 3.2  5.9 3.4
                6.1 3.4  6.1 3.2 /
\multiput{\num1} at 6.2 3.4  6.0 3 /
\multiput{\num2} at 6.2 3.2  5.8 2.8 /
\put{\num3} at 6.0 2.8 
\put{\num4} at 5.8 2.6
\multiput{\ssq} at 6.9 2.3  6.9 2.5  6.9 2.7  6.9 2.9 /
\put{\num1} at 7 2.5
\put{\num2} at 7 2.3
\multiput{\ssq} at 7.7 2.6  7.7 2.8  7.7 3  7.7 3.2  7.7 3.4  
                7.9 2.8  7.9 3  7.9 3.2  7.9 3.4   8.1 3.2  8.1 3.4 /
\multiput{\num1} at 8.2 3.4  8 3 /
\multiput{\num2} at 8.2 3.2  7.8 2.6 /
\put{\num3} at 8 2.8
\multiput{\ssq} at 8.8 2.2  8.8 2.4  8.8 2.6  8.8 2.8  8.8 3  9 2.8  9 3  /
\put{\num1} at 9.1 3
\put{\num2} at 9.1 2.8
\put{\num3} at 8.9 2.2
\multiput{\ssq} at 9.7 2.6  9.7 2.8  9.7 3  9.7 3.2  9.7 3.4 
                9.9 3  9.9 3.2  9.9 3.4  10.1 3.2  10.1 3.4 /
\multiput{\num1} at 10.2 3.4  10 3 /
\put{\num2} at 10.2 3.2
\put{\num3} at 9.8 2.6
\multiput{\ssq} at 10.9 2.4  10.9 2.6  10.9 2.8  /
\put{\num1} at 11 2.4
\multiput{\ssq} at 11.7 2.6  11.7 2.8  11.7 3  11.7 3.2  11.7 3.4  
                11.9 3.4  11.9 3  11.9 3.2       12.1 3.4 /
\multiput{\num1} at 12.2 3.4  11.8 2.8 /
\put{\num2} at 12 3
\put{\num3} at 11.8 2.6
\multiput{\ssq} at 12.8 2.2  12.8 2.4  12.8 2.6  12.8 2.8  12.8 3  13 3 /
\put{\num1} at 13.1 3
\put{\num2} at 12.9 2.4
\put{\num3} at 12.9 2.2
\multiput{\ssq} at .8 3.6  .8 3.8  .8 4  .8 4.2  .8 4.4  1 4  1 4.4  1 4.2 /
\multiput{\num1} at .9 3.8  1.1 4.2 /
\put{\num2} at 1.1 4 
\put{\num3} at .9 3.6
\multiput{\ssq} at  2.7 3.6  2.7 3.8  2.7 4  2.7 4.2  2.7 4.4  
                2.9 4.4  2.9 4  2.9 4.2      3.1 4.4 /
\multiput{\num1} at 3 4.4  3.2 4.4 /
\multiput{\num2} at 3 4.2  2.8 3.8 /
\put{\num3} at 3 4 
\put{\num4} at 2.8 3.6
\multiput{\ssq} at 4.8 3.7  4.8 3.9  4.8 4.1  4.8 4.3  5 4.3  5 4.1 /
\put{\num1} at 5.1 4.1
\put{\num2} at 4.9 3.7
\multiput{\ssq} at  6.7 3.6  6.7 3.8  6.7 4  6.7 4.2  6.7 4.4  
                6.9 3.6  6.9 3.8  6.9 4  6.9 4.2  6.9 4.4   7.1 4.2  7.1 4.4 /
\multiput{\num1} at 7.2 4.4  7 4 /
\multiput{\num2} at 7.2 4.2  6.8 3.6 /
\put{\num3} at 7 3.8
\put{\num4} at 7 3.6
\multiput{\ssq} at   9 4.3  9 4.1  8.8 3.7  8.8 3.9  8.8 4.1  8.8 4.3 /
\multiput{\num1} at 9.1 4.3  8.9 3.9 /
\put{\num2} at 9.1 4.1
\put{\num3} at 8.9 3.7
\multiput{\ssq} at    10.7 3.6  10.7 3.8  10.7 4  10.7 4.2  10.7 4.4  
                10.9 4.4   10.9 4  10.9 4.2   11.1 4.4 /
\put{\num1} at 11.2 4.4 
\put{\num2} at 11 4
\put{\num3} at 10.8 3.6
\multiput{\ssq} at 12.8 3.6  12.8 3.8  12.8 4  12.8 4.2  12.8 4.4  13 4  13 4.4  13 4.2 /
\multiput{\num1} at 12.9 3.8  13.1 4.2 /
\put{\num2} at 13.1 4 
\put{\num3} at 12.9 3.6
\multiput{\ssq} at 1.8 4.6  1.8 4.8  1.8 5  1.8 5.2  1.8 5.4  
                2 5.4  2 5  2 5.2 /
\multiput{\num1} at 2.1 5.4  1.9 4.8 /
\put{\num2} at 2.1 5.2
\put{\num3} at 2.1 5
\put{\num4} at 1.9 4.6
\multiput{\ssq} at 3.8 4.7  3.8 4.9  3.8 5.1  3.8 5.3  4 5.3 /
\put{\num1} at 4.1 5.3
\put{\num2} at 3.9 4.7
\multiput{\ssq} at 5.8 4.6  5.8 4.8  5.8 5  5.8 5.2  5.8 5.4  6 5.2  6 5.4 /
\put{\num1} at 6.1 5.2 
\put{\num2} at 5.9 4.6
\multiput{\ssq} at  7.8 4.6  7.8 4.8  7.8 5  7.8 5.2  7.8 5.4  8 5.2  8 5.4 /
\multiput{\num1} at 8.1 5.4  7.9 5 /
\put{\num2} at 8.1 5.2
\put{\num3} at 7.9 4.8
\put{\num4} at 7.9 4.6
\multiput{\ssq} at  9.8 4.7  9.8 4.9  9.8 5.1  9.8 5.3  10 5.3 /
\put{\num1} at 10.1 5.3
\put{\num2} at 9.9 4.9
\put{\num3} at 9.9 4.7
\multiput{\ssq} at 11.8 4.6  11.8 4.8  11.8 5  11.8 5.2  11.8 5.4  
                12 5.4  12 5  12 5.2 /
\put{\num1} at 12.1 5.2
\put{\num2} at 12.1 5
\put{\num3} at 11.9 4.6
\multiput{\ssq} at .8 5.6  .8 5.8  .8 6  .8 6.2  .8 6.4  1 6.4  1 6  1 6.2 /
\put{\num1} at 1.1 6.4
\put{\num2} at 1.1 6.2
\put{\num3} at 1.1 6
\put{\num4} at .9 5.6
\multiput{\ssq} at  2.9 5.7  2.9 5.9  2.9 6.1  2.9 6.3 /
\put{\num1} at 3 5.7
\multiput{\ssq} at  4.8 5.6  4.8 5.8  4.8 6  4.8 6.2  4.8 6.4  5 6.4 /
\put{\num1} at 5.1 6.4
\put{\num2} at 4.9 5.6
\multiput{\ssq} at  6.9 5.9  6.9 6.1 /
\put{\num1} at 7 5.9
\multiput{\ssq} at  8.8 5.6  8.8 5.8  8.8 6  8.8 6.2  8.8 6.4  9 6.4 /
\put{\num1} at 9.1 6.4
\put{\num2} at 8.9 6
\put{\num3} at 8.9 5.8
\put{\num4} at 8.9 5.6
\multiput{\ssq} at 10.9 5.7  10.9 5.9  10.9 6.1  10.9 6.3 /
\put{\num1} at 11 6.1
\put{\num2} at 11 5.9
\put{\num3} at 11 5.7
\multiput{\ssq} at 12.8 5.6  12.8 5.8  12.8 6  12.8 6.2  12.8 6.4  13 6.4  13 6  13 6.2 /
\put{\num1} at 13.1 6.4
\put{\num2} at 13.1 6.2
\put{\num3} at 13.1 6
\put{\num4} at 12.9 5.6
\multiput{\ssq} at   1.9 6.7  1.9 6.9  1.9 7.1  1.9 7.3 /
\multiput{\ssq} at    2.9 7.6  2.9 7.8  2.9 8  2.9 8.2  2.9 8.4 /
\multiput{\ssq} at  3.9 6.6  3.9 6.8  3.9 7  3.9 7.2  3.9 7.4 /
\put{\num1} at 4 6.6
\multiput{\ssq} at  5.9 7 /
\put{\num1} at 6 7
\multiput{\ssq} at   7.9 7 /
\multiput{\ssq} at  9.9 6.6  9.9 6.8  9.9 7  9.9 7.2  9.9 7.4 /
\put{\num1} at 10 7.2 
\put{\num2} at 10 7
\put{\num3} at 10 6.8
\put{\num4} at 10 6.6
\multiput{\ssq} at  10.9 7.6  10.9 7.8  10.9 8  10.9 8.2  10.9 8.4 /
\put{\num1} at 11 8.4
\put{\num2} at 11 8.2
\put{\num3} at 11 8
\put{\num4} at 11 7.8
\put{\num5} at 11 7.6
\multiput{\ssq} at  11.9 6.7  11.9 6.9  11.9 7.1  11.9 7.3 /
\put{\num1} at 12 7.3
\put{\num2} at 12 7.1
\put{\num3} at 12 6.9
\put{\num4} at 12 6.7
%================================
\arr{1.3 1.05} {1.7 0.45} 
\arr{2.3 0.45} {2.7 1.05} 
\arr{3.3 1.05} {3.7 0.45} 
\arr{4.3 0.45} {4.7 1.05} 
\arr{5.3 1.05} {5.7 0.45} 
\arr{6.3 0.45} {6.7 1.05} 
\arr{7.3 1.05} {7.7 0.45} 
\arr{8.3 0.45} {8.7 1.05} 
\arr{9.3 1.05} {9.7 0.45} 
\arr{10.3 0.45} {10.7 1.05} 
\arr{11.3 1.05} {11.7 0.45} 
\arr{12.3 0.45} {12.7 1.05} 
%=================================
\arr{1.3 1.95} {1.6 2.4} 
\arr{2.3 2.55} {2.7 1.95} 
\arr{3.3 1.95} {3.6 2.4} 
\arr{4.3 2.55} {4.7 1.95} 
\arr{5.3 1.95} {5.7 2.55} 
\arr{6.4 2.4 } {6.7 1.95} 
\arr{7.3 1.95} {7.6 2.4 } 
\arr{8.4 2.4 } {8.7 1.95} 
\arr{9.3 1.95} {9.7 2.55} 
\arr{10.3 2.55} {10.7 1.95} 
\arr{11.3 1.95} {11.7 2.55} 
\arr{12.3 2.55} {12.7 1.95} 
%================================
\arr{1.3 3.7} {1.6 3.4} 
\arr{2.3 3.3} {2.6 3.6} 
\arr{3.3 3.7} {3.6 3.4} 
\arr{4.4 3.4} {4.7 3.7} 
\arr{5.3 3.7} {5.6 3.4} 
\arr{6.4 3.4} {6.6 3.6} 
\arr{7.4 3.6} {7.6 3.4} 
\arr{8.4 3.4} {8.7 3.7} 
\arr{9.3 3.7} {9.7 3.3} 
\arr{10.4 3.4} {10.7 3.7} 
\arr{11.4 3.6} {11.7 3.3} 
\arr{12.4 3.4} {12.7 3.7} 
%=================================
\arr{1.3 4.3} {1.7 4.7} 
\arr{2.3 4.7} {2.6 4.4} 
\arr{3.3 4.3} {3.7 4.7} 
\arr{4.3 4.7} {4.7 4.3} 
\arr{5.3 4.3} {5.7 4.7} 
\arr{6.3 4.7} {6.6 4.4} 
\arr{7.4 4.4} {7.7 4.7} 
\arr{8.3 4.7} {8.7 4.3} 
\arr{9.3 4.3} {9.7 4.7} 
\arr{10.3 4.7} {10.7 4.3} 
\arr{11.4 4.4} {11.7 4.7} 
\arr{12.3 4.7} {12.7 4.3} 
%================================
\arr{1.3 5.7} {1.7 5.3} 
\arr{2.3 5.3} {2.7 5.7} 
\arr{3.3 5.7} {3.7 5.3} 
\arr{4.3 5.3} {4.7 5.7} 
\arr{5.3 5.7} {5.7 5.3} 
\arr{6.3 5.3} {6.7 5.7} 
\arr{7.3 5.7} {7.7 5.3} 
\arr{8.3 5.3} {8.7 5.7} 
\arr{9.3 5.7} {9.7 5.3} 
\arr{10.3 5.3} {10.7 5.7} 
\arr{11.3 5.7} {11.7 5.3} 
\arr{12.3 5.3} {12.7 5.7} 
%================================
\arr{1.3 6.3} {1.7 6.7} 
\arr{2.3 6.7} {2.7 6.3} 
\arr{3.3 6.3} {3.7 6.7} 
\arr{4.3 6.7} {4.7 6.3} 
\arr{5.3 6.3} {5.7 6.7} 
\arr{6.3 6.7} {6.7 6.3} 
\arr{7.3 6.3} {7.7 6.7} 
\arr{8.3 6.7} {8.7 6.3} 
\arr{9.3 6.3} {9.7 6.7} 
\arr{10.3 6.7} {10.7 6.3} 
\arr{11.3 6.3} {11.7 6.7} 
\arr{12.3 6.7} {12.7 6.3} 

\arr{2.3 7.3} {2.7 7.7} 
\arr{3.3 7.7} {3.7 7.3} 
\arr{10.3 7.3} {10.7 7.7} 
\arr{11.3 7.7} {11.7 7.3} 

%================================
\arr{1.3 2.72}{1.6 2.84}
\arr{2.4 2.84}{2.7 2.72}
\arr{3.3 2.72}{3.6 2.84}
\arr{4.4 2.84}{4.7 2.72}
\arr{5.3 2.72}{5.6 2.84}
\arr{6.4 2.84}{6.7 2.72}
\arr{7.3 2.72}{7.6 2.84}
\arr{8.4 2.84}{8.7 2.72}
\arr{9.3 2.72}{9.6 2.84}
\arr{10.4 2.84}{10.7 2.72}
\arr{11.3 2.72}{11.6 2.84}
\arr{12.4 2.84}{12.7 2.72}
\setdots<2pt>
%\plot 0.7 0  1.3 0 /
\plot 1 0  1.7 0 /
\plot 2.3 0  3.7 0 /
\plot 4.3 0  5.7 0 /
\plot 6.3 0  7.7 0 /
\plot 8.3 0  9.7 0 /
\plot 10.3 0  11.7 0 /
\plot 12.3 0  13 0 /
\plot 1.3 2.6  1.6 2.6 /
\plot 2 2.6  2.7 2.6 /
\plot 3.3 2.6  3.6 2.6 /
\plot 4 2.6  4.7 2.6 /
\plot 5.3 2.6  5.6 2.6 /
\plot 6 2.6  6.7 2.6 /
\plot 7.3 2.6  7.6 2.6 /
\plot 8 2.6  8.7 2.6 /
\plot 9.3 2.6  9.6 2.6 /
\plot 10 2.6  10.7 2.6 /
\plot 11.3 2.6  11.6 2.6 /
\plot 12 2.6  12.7 2.6 /
%\plot 0.7 7  1.3 7 /
\plot 1 7  1.7 7 /
\plot 4.3 7  5.7 7 /
\plot 6.3 7  7.7 7 /
\plot 8.3 7  9.7 7 /
\plot 12.3 7  13 7 /
\setsolid
\plot 1 0  1 1.1 /
\plot 1 1.9  1 2 /
\plot 1 3.2  1 3.4 /
\plot 1 4.6  1 5.4 /
\plot 1 6.6  1 7 /
\plot 13 0  13 1.1 /
\plot 13 1.9  13 2 /
\plot 13 3.2  13 3.4 /
\plot 13 4.6  13 5.4 /
\plot 13 6.6  13 7 / }

\begin{figure}[hb]
$$
\hbox{\beginpicture
\setcoordinatesystem units <.95cm,1.1cm>
%==========================================
\GammaFive
\endpicture}
$$
\end{figure}

\medskip
Recall that for $R=k[[T]]$ the power series ring, the homomorphisms between indecomposable
modules are given as the linear combinations of paths, modulo mesh relations.

\smallskip
In the second copy of $\Gamma_{\mathcal S(5)}$, the pickets are encircled, and there is an
indecomposable object labelled $M$ which has the following LR-tableau.
$$
\hbox{\beginpicture
\setcoordinatesystem units <1.3cm,1.3cm>
%==========================================
\put{$M:$} at .5 .5 
\multiput{\ssq} at 1.9 .1  1.9 .3  1.9 .5  1.9 .7  1.9 .9  2.1 .5  2.1 .7  2.1 .9  2.3  .9 /  
\multiput{\num1} at 2.2 .7  2.4  .9 /
\multiput{\num2} at 2 .3  2.2 .5 /
\put{\num3} at 2 .1 
\endpicture}
$$

For each pair $(\ell,m)$ such that 
$$\frac{\Hom_{\mathcal S}(M,P_\ell^m)}{\Im\Hom_{\mathcal S}(M,g_\ell^m)} \;\neq \;0$$
we indicate a path $M\to P_\ell^m$ representing a map which does not factor through $g_\ell^m$.
There are the following five paths:
$$ M\to P_1^1,\quad  M\to P_1^2, \quad M\to P_2^3, \quad M\to P_2^4, \quad\text{and}\quad 
   M\to P_3^5$$
Corresponding to each such path $M\to P_\ell^m$ there is an entry $\singlebox{\ell}$
in the $m$-th row of the LR-tableau for $M$, as predicted by Theorem~\ref{theorem-picket}.

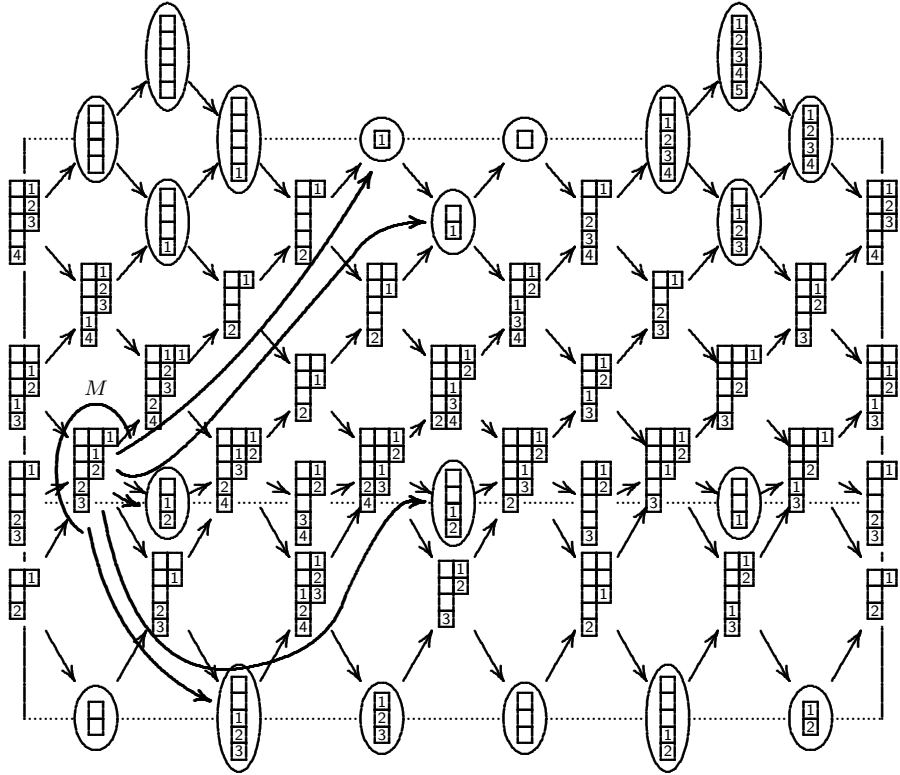
\begin{figure}[hb]
$$
\hbox{\beginpicture
\setcoordinatesystem units <.95cm,1.1cm>
%==========================================
\GammaFive
\ellipticalarc axes ratio 2:3 360 degrees from 2.3 0 center at 2 0
\ellipticalarc axes ratio 2:5 360 degrees from 4.3 0 center at 4 0
\ellipticalarc axes ratio 3:5 360 degrees from 6.3 0 center at 6 0
\ellipticalarc axes ratio 3:5 360 degrees from 8.3 0 center at 8 0
\ellipticalarc axes ratio 2:5 360 degrees from 10.3 0 center at 10 0
\ellipticalarc axes ratio 2:3 360 degrees from 12.3 0 center at 12 0
\ellipticalarc axes ratio 1:2 360 degrees from 2.3 7 center at 2 7
\ellipticalarc axes ratio 2:5 360 degrees from 4.3 7 center at 4 7
\ellipticalarc axes ratio 1:1 360 degrees from 6.3 7 center at 6 7
\ellipticalarc axes ratio 1:1 360 degrees from 8.3 7 center at 8 7
\ellipticalarc axes ratio 2:5 360 degrees from 10.3 7 center at 10 7
\ellipticalarc axes ratio 1:2 360 degrees from 12.3 7 center at 12 7
\ellipticalarc axes ratio 2:5 360 degrees from 3.3 8 center at 3 8
\ellipticalarc axes ratio 1:2 360 degrees from 3.3 6 center at 3 6
\ellipticalarc axes ratio 2:5 360 degrees from 11.3 8 center at 11 8
\ellipticalarc axes ratio 1:2 360 degrees from 11.3 6 center at 11 6
\ellipticalarc axes ratio 2:3 360 degrees from 7.3 6 center at 7 6
\ellipticalarc axes ratio 3:5 360 degrees from 3.3 2.6 center at 3 2.6
\ellipticalarc axes ratio 1:2 360 degrees from 7.3 2.6 center at 7 2.6
\ellipticalarc axes ratio 3:5 360 degrees from 11.3 2.6 center at 11 2.6
\put{$\scriptstyle M$} at 2 4
\linethickness=6pt
\setplotsymbol({.})
\ellipticalarc axes ratio 3:5 160 degrees from 2 3.8 center at 2 3
\ellipticalarc axes ratio 3:5 -60 degrees from 2 3.8 center at 2 3
\arr{2.5 2.60}{2.6 2.57}
\arr{3.5 .245}{3.6 .21}
\arr{6.5 2.62}{6.6 2.62}
\arr{5.78 6.5}{5.85 6.6}
\arr{6.5 6}{6.6 6}
\setquadratic
\plot 2.25 2.72  2.5 2.6  2.6 2.6 /
\plot 1.9 2.3  2.5 1  3.6 .2 /
\plot 2.1 2.5  2.8 1  4 .6  5 .9  5.5 1.5  6.1 2.43  6.5 2.6 /
\plot 2.3 3.2  4.1 4.5  5.85 6.6 /
\plot 2.3 3  2.8 3  3.6 3.6  4.6 4.6  5.6 5.6  6 5.9  6.6 6 /
\endpicture}
$$
\end{figure}

\smallskip
Note that nonisomorphic objects in $\mathcal S$ can have the same LR-tableau.
Consider for example the Auslander-Reiten sequence
$$0\to C_2^4 \to M\to P_2^3 \to 0$$
with middle term $M$.  The modules $M$ and $C_2^4\oplus P_2^3$ have the same LR-tableau,
and hence cannot be distinguished by homomorphisms into pickets.

\smallskip
We focus on the case where $\ell=2$ and $m=4$.  The indecomposables which have
an entry $\singlebox{2}$ in the $4$-th row in their LR-tableau are in the region
labelled $\mathcal R$ in the third copy of the Auslander-Reiten quiver 
$\Gamma_{\mathcal S(5)}$.
Note that the two ``eyes'' are not part of the region $\mathcal R$.
Each object $M$ in $\mathcal R$ 
admits a map $t: M\to P_2^4$ which does not factor through $g_2^4$.
(In the diagram, the module $P_2^4$ is labelled ``$Z$'', while the summands 
$P_1^4$ and $P_2^3$ of the 
source of $g_2^4$ are labelled ``$Y_1$'' and ``$Y_2$''.)
According to Proposition~\ref{prop-bilinear}, those modules admit a map $t':C_2^4\to M$
such that the composition $tt'$ does not factor through $g_2^4$. 
The module $C_2^4$ is obtained as follows. Let $\mathcal E$ be the short exact sequence
given by $g_2^4$:
$$\mathcal E: 0\quad\longrightarrow \quad P_1^3
                     \stackrel{f_2^4}{\longrightarrow}\quad
                   P_1^4 \;\oplus\; P_2^3 \quad
                     \stackrel{g_2^4}{\longrightarrow}\quad
                   P_2^4 \quad\longrightarrow\quad 0 $$
and put $C_2^4=\tau_{\mathcal S(5)}^{-1}(P_1^3)$ (in the diagram, the modules 
$C_2^4$ and $P_1^3$ are labelled ``C'' and ``$X$'', respectively).
As predicted by the Proposition, 
the indecomposables in $\mathcal S(5)$ which have a $\singlebox{2}$ in the
$4$-th row are located between $C_2^4$ and $P_2^4$.

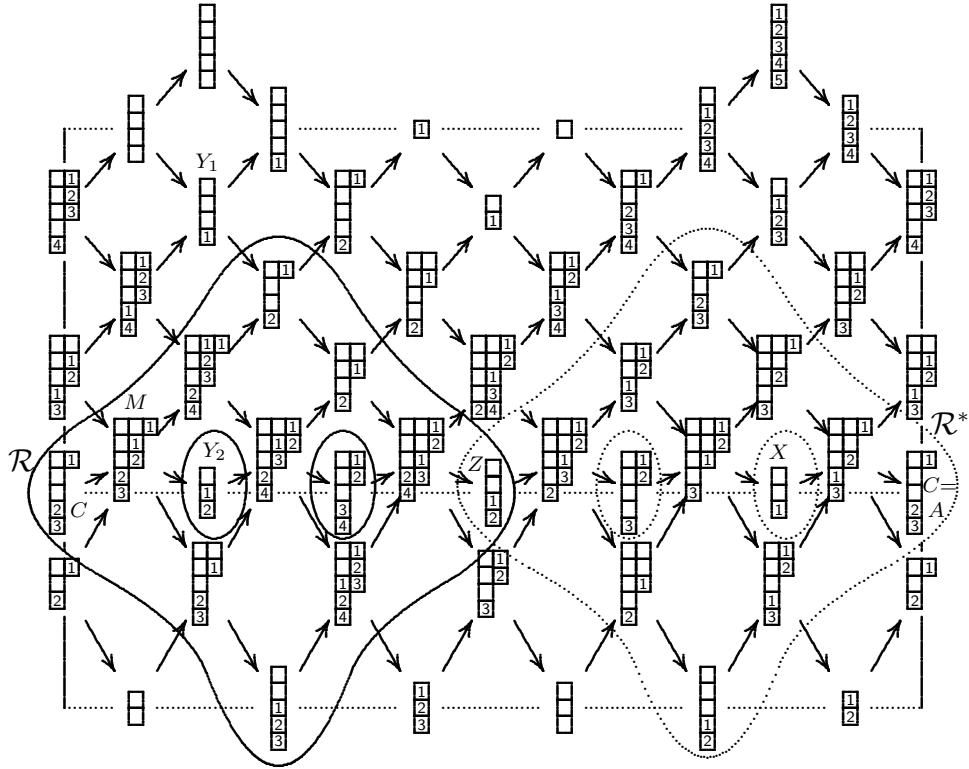
\begin{figure}[ht]
$$
\hbox{\beginpicture
\setcoordinatesystem units <.95cm,1.1cm>
%==========================================
\GammaFive
\setsolid
\setquadratic
\plot 1.5 1.6  .5 2.6  1.5 3.6  2.3 4.2  3 5  4 5.7  5 5  5.7 4.2  6.5 3.6  7.3 2.6  
      6.5 1.6  5.5 1  5 .4  4 -.7  3 .4  2.5 1  1.5 1.6 /
%\multiput{\bul} at 1.5 1.6  .5 2.6  1.5 3.6  2.3 4.2  3 5  4 5.7  5 5  5.7 4.2  6.5 3.6  7.5 2.6  
%      6.5 1.6  5.5 1  5 .4  4 -.7  3 .4  2.5 1  1.5 1.6 /
\ellipticalarc axes ratio 3:5 360 degrees from 4.9 3.35 center at 4.9 2.7
\ellipticalarc axes ratio 3:5 360 degrees from 3.1 3.35 center at 3.1 2.7
\put{$\scriptstyle M$} at 2 3.7
\put{$\mathcal R$} at .4 3
\put{$\scriptstyle Y_2$} at 3.1 3.1 
\put{$\scriptstyle Y_1$} at 3.0 6.6
\put{$\scriptstyle Z$} at 6.75 2.95
\put{$\scriptstyle X$} at 11 3.1
\put{$\scriptstyle C$} at 1.2 2.4
\put{$\scriptstyle C=$} at 13.25 2.7
\put{$\scriptstyle A$} at 13.2 2.4
\put{$\mathcal R^*$} at 13.4 3.45
\setdots<2pt>
\setquadratic
\plot 7.5 1.7  6.5 2.7  7.5 3.7  8.3 4.3  9 5.1  10 5.8  11 5.1  11.7 4.3  12.5 3.7  13.5 2.7  
      12.5 1.7  11.5 1.1  11 .5  10 -.6  9 .5  8.5 1.1  7.5 1.7 /
\ellipticalarc axes ratio 3:5 360 degrees from 11.1 3.35 center at 11.1 2.7
\ellipticalarc axes ratio 3:5 360 degrees from 8.9 3.35 center at 8.9 2.7
\endpicture} 
$$
\end{figure}

\bigskip
We consider duality:

\smallskip
Let us locate those indecomposable objects $M\in\mathcal S(5)$ for which the 
LR-tableau of $M^*$ contains a box $\singlebox 2$ in row 4. 
Note that duality acts on the above Auslander-Reiten quiver 
by reflection on the central vertical axis.  
It follows that the modules $M$ as above are located within the region $\mathcal R^*$
encircled by the dotted line, and without the two dotted ellipses. 
According to Theorem~\ref{theorem-picket-dual}, 
they are characterized in terms of homomorphisms from pickets, as follows.
The dual of the above sequence $\mathcal E$ is 
$$\setcounter{boxsize}{3}
\mathcal E^*: 0\quad\longrightarrow \quad P_2^4 \quad
                     \stackrel{h_2^4}{\longrightarrow}\quad
                   P_3^4 \;\oplus\;
                   P_1^3 \quad \longrightarrow\quad P_2^3 \quad\longrightarrow\quad 0.$$

\bigskip
The modules in $\mathcal R^*$ admit a map $t:P_2^4\to M$ which does not factor over $h_2^4$.
Let $A=A_2^4=\tau_{\mathcal S(5)}(P_2^3)$ be the $\tau$-translate 
of the cokernel of $h_2^4$; it is indicated at the right end of the region $\mathcal R^*$.
According to Proposition~\ref{prop-bilinear-dual} 
the modules $M$ in $\mathcal R^*$ are between $P_2^4$ and $A_2^4$ in the sense that given $t$, 
there is a map $t'':M\to A_2^4$ such that $t''t$ does not 
factor through $h_2^4$.

\medskip
Address of the author: 
Mathematical Sciences, 
Florida Atlantic University,
Boca Raton, Florida 33431-0991

\smallskip
E-mail: {\tt markusschmidmeier@gmail.com}

\end{document}